\documentclass[14pt]{amsart}
\usepackage{hyperref}
\usepackage[usenames]{color}
\usepackage{amsmath,amsthm,amsfonts,amssymb}

\newtheorem{theorem}{Theorem}

\newtheorem{proposition}[theorem]{Proposition}

\theoremstyle{definition}

\theoremstyle{remark}
\newtheorem{remark}[theorem]{Remark}

\numberwithin{equation}{section}

\begin{document}

\title{The stabilizer of a column in a  matrix group over a polynomial  ring}

\author{Vitaly Roman'kov}\footnote{ This research was supported by the program of basic scientific researches SB RAS I.114, project 0314-2019-0004.}

\begin{abstract}
An original non-standard approach to describing the structure of a column stabilizer in a group of $n \times n$ matrices over a polynomial ring or a Laurent polynomial ring of $n$ variables is presented. The stabilizer is described as an extension of a subgroup of a rather simple structure using the  $(n-1) \times (n-1)$ matrix group of congruence type over the corresponding ring of $n-1$ variables. In this paper, we consider cases where $n \leq 3.$ For $n = 2$, the stabilizer is defined as a one-parameter subgroup, and the proof is carried out by direct calculation. The case $n = 3$ is nontrivial; the approach mentioned above is applied to it. Corollaries are given to the results obtained. 
In particular, we prove that for the stabilizer in the question, it is not generated by its a finite subset together with the so-called tame stabilizer of the given column. We are going  to study the cases when $n \geq 4$ in a forthcoming  paper. Note that a number of key subgroups of the groups of automorphisms of groups are defined as column stabilizers in matrix groups. For example, this describes the subgroup IAut($M_r$) of automorphisms that are identical modulo a commutant of a free metabelian group $M_r$ of rank $r$. This approach demonstrates the parallelism of theories of groups of automorphisms of groups and matrix groups that exists for a number of well-known groups. This allows us to use the results on matrix groups to describe automorphism groups. In this work, the classical theorems of Suslin, Cohn, as well as Bachmuth and Mochizuki are used.

\end{abstract}

\maketitle

{\small  \tableofcontents}

 \section*{Introduction}

In the group theory, matrix methods have been used by a number of authors to produce new interesting results on endomorphisms and automorphisms of groups. Birman \cite{Bir} has given a matrix characterization of automorphisms of a free group $F_r$ of rank $r$ with basis $\{f_1, ...,  f_r\}$ among arbitrary endomorphisms (the "inverse function theorem") as follows. For an endomorphism $\phi$ define the matrix $J_{\phi}$ = ($d_j\phi (f_i)$), $1\leq i, j\leq r$ (the "Jacobian matrix"{} of $\phi$), where $d_j$ denotes partial Fox derivation (with respect to $f_j$) in the free group ring $ \mathbb{Z}[F_r]$ (see \cite{Ess} or \cite{TimBook}). Then $\phi$ is an automorphism if and only if the matrix $J_{\phi}$ is invertible. 
 
Bachmuth \cite{Bach} has obtained an inverse function theorem of the same kind on replacing the Jacobian matrix $J_{\phi}$ by its image $\bar{J}_{\phi}$  over the abelianized group ring $\mathbb{Z}[F_r/F_r'].$ Thus he established a matrix characterization of automorphisms of a free metabelian group $M_r$. Umirbaev \cite{U}  has generalized Birman's result to primitive systems of free groups,    Roman'kov \cite{RomCrit}, \cite{Prim3} and Timoshenko \cite{TimCrit} have characterized  primitive systems of free metabelian groups. By definition, {\it primitive system} is a system of elements of a relatively free group that can be a part of some basis of this group.

For any commutative associative ring $K$ with identity, an $r\times r$ elementary matrix (transvection) $t_{ij}(a)$ over $K$ is a matrix of the form $E +aE_{ij}$ where $i\not=  j, a \in K$,  $E_{ij}$ is the $r \times r$ matrix whose $(i j)$ component is $1$ and all other components are zero. As usual, $E$ denotes the identity matrix. 
Let SL($r, K$) be  the group of all the $r\times r$ matrices of determinant $1$ whose entries are elements of $K$, and let E($r, K$)  be the subgroup of SL($r, K$) generated by the elementary matrices. 
By $\Lambda_{nk}^{K} = K[a_1, ..., a_k,  a_{k+1}^{\pm 1}, ...,  a_n^{\pm 1}]$ we denote a 
mixed polynomial ring over $K$. In particular,  $\Lambda_{nn}^{K} = K[a_1, ..., a_n]$ is the polynomial ring and $\Lambda_{n0}^{K} = K[a_1^{\pm 1}, ..., a_n^{\pm 1}]$ is the Laurent polynomial ring in $n$ variables over $K.$ 
 
Then the famous Suslin's Stability theorem \cite{Sus} implies that for any $r\geq 3$ and any ring $\Lambda_{nk}^{\mathbb{F}}$,  where  $\mathbb{F}$ is an arbitrary field,  SL($r, \Lambda_{nk}^{\mathbb{F}}$) = E($r, \Lambda_{nk}^{\mathbb{F}}$).  

By GE($r, K$) we denote  the subgroup  of GL($r, K$) generated by E($r, K$) and all diagonal matrices. It follows that for any $r\geq 3$ and any ring $\Lambda_{nk}^{\mathbb{F}}$, GL($r, \Lambda_{nk}^{\mathbb{F}}$) = GE($r, \Lambda_{nk}^{\mathbb{F}}$).

In contrast, GL($2, \Lambda_{nk}^{\mathbb{F}}$) has a number of specific properties. In \cite{Cohn}, Cohn proved that
\begin{equation}
\label{eq:1} 
\left(\begin{array}{cc}
1 +a_1a_2& -a_1^2\\
a_2^2&1-a_1a_2\\
\end{array}\right) \in \rm{GL}(2, \Lambda_{22}^{\mathbb{F}})\setminus \rm{GE}(2, \Lambda_{22}^{\mathbb{F}}).
\end{equation} 
In \cite{BM}, Bachmuth and Mochizuki proved that  if $n\geq 2$ then    
\begin{equation}
\label{eq:2}
\rm{GL}(2,\Lambda_{n0}^{\mathbb{Z}}) \not= \rm{GE}(2, \Lambda_{n0}^{\mathbb{Z}}). 
\end{equation}
Let $M_r$ be the free metabelian group of rank $r$ with basis $\{x_1, ..., x_r\},$  and $A_r = M_r/M_r'$ be the abelianization of $M_r$, the free abelian group with the corresponding basis $\{a_1, ..., a_r\}.$
The group ring $\mathbb{Z}[A_r]$ can be considered as the Laurent polynomial ring $\Lambda_{r0}.$  

For any group $G$, IAut($G$) denotes the subgroup of the automorphism group Aut($G$) consisting of all automorphisms that induce the identity map on the abelianization $G_{ab}=G/G'.$ In the similar way the subsemigroup IEnd($G$) of the endomorphism semigroup End($G$) is defined too. 

In \cite{Bach}, Bachmuth introduced the following embedding:
\begin{equation}
\label{eq:3}
\beta : \rm{IAut}(M_r) \rightarrow  \rm{GL}(r, \Lambda_{r0}^{\mathbb{Z}}), \beta : \phi \mapsto \bar{J}_{\phi}, \phi \in \rm{IAut}(M_r).
\end{equation}
 This  embedding is called {\it Bachmuth's embedding}.

The  image  $\beta$(IAut($M_r$)) 
in GL$_r( \Lambda_{r0}^{\mathbb{Z}})$ consists of all matrices $A$ such that
\begin{equation}
\label{eq:4}
 A\bar{a} = \bar{a}_r  \  \textrm{for} \  \bar{a}_r = \left( \begin{array}{c}
                                    a_1-1\\
a_2-1\\
.\\
.\\
.\\
a_r-1
\end{array}
\right). 
\end{equation}
In other words, IAut($M_r$) = Stab$_{\rm{GL}(r, \Lambda_{r0}^{\mathbb{Z}})}$
($\bar{a}_r$) 
(the stabilizer of $\bar{a}_r$ in GL($r, \Lambda_{r0}^{\mathbb{Z}}$)). Thus, this is an example showing that key subgroups can act as column stabilizers in matrix groups.  In \cite{Shp}, Shpilrain obtained a  matrix characterization   of IA-endomorphisms with non-trivial fixed points  ('eigenvectors') which, although is similar to the corresponding well-known characterization in linear algebra, also reveals a subtle difference. All these and some other results show a wonderful parallelism between the theory of automorphisms and endomorphisms of a free (or free metabelian) group and the theory of linear operators in vector space.

The main goal of this paper is to present an original non-standard approach to the description of column stabilizers in matrix groups over rings. We consider matrix groups over polynomial rings $\Lambda_{nn}^{\mathbb{K}}$ and matrix groups over Laurent polynomial rings $\Lambda_{n0}^{\mathbb{K}}.$ In both cases $K$ is an arbitrary commutative domain with identity element.  For simplicity, we formulate some of statements  only in the following important  cases: $K = \mathbb{Z}$ or $\mathbb{F},$ where $\mathbb{F}$ is an arbitrary field.

In this paper we consider only cases of $n\times n$ matrices for $n \leq 3$. The case $n=3$ is the least non-trivial in the subject. In the forthcoming paper we'll extend our method to the description  of   column stabilizers for the cases $n \geq 4.$ We also restrict ourselves to considering stabilizers of columns of a certain type -- either columns of variables for rings of polynomials, or columns with components of the form "a variable minus $1$"{} for Laurent polynomials. 

For the case $n = 2$, we give an exhaustive description of the stabilizer of a column  as a one-parameter subgroup.
For $n=3$  we describe a stabilizer of a column  as extension of a  subgroup with a simple structure  by a specific group of congruence type of $2\times 2$ matrices over ring on $2$ variables. The idea of such description was originated in \cite{RomRes} and \cite{RomAutmet}. Such a description was successfully used in \cite{RomAutmet}  to prove that every automorphism of $M_r, r\geq 4,$ is induced by an automorphism of $F_r$, i.e., is tame.  Also this description was used in \cite{Prim3} to prove that $M_3$ contains primitive elements that are not images of primitive elements of $F_3.$  

At the last Section 3 we derive a number of corollaries of the obtained  results about stabilizers of columns in 
the case $n=3$. 

\begin{remark}
\label{re:1}
The column stabilizer in a $n\times n$ matrix group over a  field can be described as follows. 
Having included the stabilized vector  as the last element of the basis of the corresponding linear space, we get each of the stabilizer matrices in the half-expanded form when the last column is of the form $\left(\begin{array}{c}
0\\ ... \\0\\ 1\end{array}\right).$ The stabilizer consists of all matrices of the such form. It has as a homomorphic image the corresponding group of $(n-1)\times (n-1)$ matrices with the kernel of obvious structure. A similar description for matrix group over a ring is possible if at least one component of the stabilized column is invertible. See (\ref{eq:12}) below. If $c_3$ is invertible one has a homomorphism as above. Our approach is  useful for other cases. 
\end{remark}

\section{Preliminaries}
 
Let $K$ be an arbitrary commutative associative domain with identity. 
 For any $n \in \mathbb{N}$, let $\Lambda_n^{K}$ denotes the polynomial ring $\Lambda_{nn}^{\mathbb{K}}$ or the  Laurent polynomial ring $\Lambda_{n0}^{\mathbb{K}}.$  Let $\Delta_n^{K}$ stays for id($a_1, ..., a_n$) of $\Lambda_{nn}^{K}$ or  for id($a_1-1, ..., a_n-1$) of $\Lambda_{nn}^{K}$ (the augmentation ideal of $\Lambda_n^{K}$). Denote $c_i = a_i$ in the case of $\Lambda_n^K=\Lambda_{nn}^{K},$ and  $c_i= a_i-1$ in the case of $\Lambda_{nn}^{K}$. Further in the paper, we'll omit $K$ for brevity and simply write $\Lambda_n.$

Each element  $g\in \Lambda_k, k \leq n,$ has for every $t\geq 1$ the unique expression of the form 
\begin{equation} 
\label{eq:5}
g = \sum_{i=0}^tg_ic_k^i, 
\end{equation}
\noindent where $g_i\in \Lambda_{k-1}$ for $i = 0, ..., t-1$, and $g_t\in \Lambda_k.$ 
 Since every ring $\Lambda_k$ embeds into a field of fractions, we can consider a $\Lambda_k$-submodule $\Lambda_k^{(-)}=\Lambda_k+ c_k^{-1}\Lambda_k,$ and each element of $\Lambda_k^{(-)}$ has for each $t\geq 0$ the unique expression of the form 
\begin{equation}
\label{eq:6}
g = \sum_{i = -1}^tg_ic_k^i, 
\end{equation}
\noindent where $g_i\in \Lambda_{k-1}$ for $i = -1, ..., t-1$, and $g_t\in \Lambda_k.$

Denote 
\begin{equation}
\label{eq:7}
\bar{c}_n = \left(\begin{array}{c}
c_1\\
c_2\\...
\\
c_n\end{array}\right).
\end{equation}
All along the paper we assume  $n = 3$ (with one short exception for $n=2$ at the beginning of the next Section 2).  Let $G = $ Stab($\bar{c}_3)$ be the subgroup of GL($3, \Lambda_3$) consisting of all matrices $g$  such that
\begin{equation}
\label{eq:8}
g\bar{c}_3 = \bar{c}_3,
\end{equation}
\noindent
in other words, $G$ is the stabilizer of the column $\bar{c}_3$ in the group GL($3,\Lambda_3$). We'll show how to construct an explicit  matrix group $H\leq$ GL($2, \Lambda_2$) and a homomorphism $\rho$ of $G$ onto $H$ for which ker($\rho$) is well understood. In other words, we'll describe $G$ as an extension ker($\rho$) by im($\rho$) with explicitly desribed factors. We'll give a number of applications of these results.

\section{On the stabilizer of a column in   GL($3, \Lambda_3$) }

Before considering the case of $3\times 3$  matrices, we show how the stabilizer of the vector $\bar{c}_2$  is arranged in the group of $2\times 2 $  matrices over $\Lambda_2$.

\begin{proposition} 
\label{pr:1}
In \rm{GL}($2, \Lambda_2$), 
\begin{equation}
\label{eq:9}
\rm{Stab}(\bar{c}_2)  = \{\left(\begin{array}{cc}1+ac_1c_2& -ac_1^2\\
ac_2^2& 1-ac_1c_2\\
\end{array}
\right)\}, 
\end{equation}
\noindent where  $a\in \Lambda_2.$
\end{proposition}
\begin{proof}
Obviously, every matrix $A$ in M($2, \Lambda_2$)  such that $A\bar{c}_2= \bar{c}_2$ has the form 
\begin{equation}
\label{eq:10}
\left(\begin{array}{cc}1+ac_2& -ac_1\\
bc_2& 1-bc_1\\
\end{array}
\right).
\end{equation}
A matrix of the form (\ref{eq:10}) is invertible if and only if its determinant is $1$. By direct computation we obtain that this happens if and only if this matrix has the form (\ref{eq:9}). 
\end{proof}

Now $c_1, c_2, c_3$ are three pairwise non-associated prime elements of $\Lambda_3$ such that each element $g \in \Lambda_3$ can be uniquely expressed in the form
\begin{equation} 
\label{eq:11}
g = \sum_{i=0}^2g_ic_3^i, 
\end{equation}
\noindent
where $g_0, g_1 \in \Lambda_2$ and $g_2\in \Lambda_3$. Let $G$ is the stabilizer of the column $\bar{c}_3$ in the group GL($3,\Lambda_3$). Denote 
\begin{equation}
\label{eq:12}
C =\left(\begin{array}{ccc}
1& 0 & c_1\\
0& 1 & c_2\\
0& 0 & c_3\end{array}\right).
\end{equation}
For $A= (a_{ij})\in G$ we have the following equality
\begin{equation}
\label{eq:13}
C^{-1}AC = \left(\begin{array}{ccc}
1+a_{11}-a_{31}c_1c_3^{-1}& a_{12} - a_{32}c_1c_3^{-1} & 0\\
a_{21} - a_{31} c_2c_3^{-1}&1+a_{22}-a_{32}c_2c_3^{-1}   & 0\\
a_{31}c_3^{-1}&a_{32}c_3^{-1}  & 1
\end{array}\right).
\end{equation}

Denote 
\begin{equation}
\label{eq:14}R(A) = \left(\begin{array}{cc}
1+a_{11}-a_{31}c_1c_3^{-1}& a_{12} - a_{32}c_1c_3^{-1}\\
a_{21} - a_{31}c_2c_3^{-1}&1+a_{22}-a_{32}c_2c_3^{-1}\\
\end{array}\right) \in \rm{GL}(2, \Lambda_3^{(-)}).
\end{equation}
Then we have homomorphism 
\begin{equation}
\label{eq:15}
\theta : G \rightarrow \rm{GL}(2, \Lambda_3^{(-)}), \  \theta : A \mapsto R(A).
\end{equation}
Using (\ref{eq:6}), we obtain a decomposition of the form
\begin{equation}
\label{eq:16}
R=R(A) = E + R_2c_3^2+R_1c_3+ R_0+R_{-1}c_3^{-1},
\end{equation}
\noindent
where
\begin{equation}
\label{eq:17}
R_1, R_0, R_{-1}\in \textrm{M}_2(\Lambda_2),  R_{2}\in \textrm{M}_2(\Lambda_3).
\end{equation}
We put
\begin{equation}
\label{eq:18}
X = \left(\begin{array}{cc}
c_1c_2&-c_1^2\\
c_2^2&-c_1c_2\\
\end{array}\right).
\end{equation}
\begin{theorem}
\label{th:1}
In the above notation, there exist elements $\alpha , \beta , \gamma , \delta \in \Lambda_2$  such that
\begin{equation}
\label{eq:19}
R_{-1} = \alpha X, R_0 X = \beta X, XR_0 = \gamma X, XR_1X = \delta X.
\end{equation}
\end{theorem}
\begin{proof}
Since $A\in G$, we have $a_{31}c_1+ a_{32}c_2+ a_{33}c_3=c_3.$ Then $(a_{31})_0c_1+ (a_{32})_0=0,$ and so $(a_{31})_0 = - \alpha c_2, (a_{32})_0 = \alpha c_1$
for some $\alpha \in \Lambda_2.$ It follows, that $R_{-1}=\alpha X.$
 
Note that $ T = E - e_{31}c_2 + e_{32}c_1\in G$ and $R(T) = E + Xc_3^{-1}.$ It follows that  
$R(A)R(T)$ has the $(-1)$-component $R_0X$ and so $R_0X= \beta X, \beta \in \Lambda_2.$ Similarly, $R(T)R(A)$ has the $(-1)$-component $XR_0$, hence $XR_0= \gamma X, \gamma X, \gamma \in \Lambda_2.$ At last, $R(T)R(A)R(T)$ has the $(-1)$-component $XR_1X$, hence $XR_1X= \delta X,  \delta \in \Lambda_2.$ 
\end{proof}
Thus,  we can associate the elements $\alpha , \beta , \gamma , \delta \in \Lambda_2$ with the matrix $A\in G$.  These elements are called {\it residues} of $A$ and of  $R$ with respect to $c_3$ Now we give explicit formulas for the residues  in terms of elements of the matrices $R_i, i = 1, 0, -1.$   These formulas are obtained by direct computations. 

\begin{equation}
\label{eq:20}
\alpha = -(a_{31})_0c_2^{-1}= (a_{32})_0c_1^{-1}, \beta = -(a_{31})_1c_1 - (a_{32})_1c_2, \gamma = (a_{11})_0 - (a_{21})_0c_1c_2^{-1},\end{equation}
$$  \delta = - a_{21})_1c_1^2 + (a_{12})_1c_2^2 + ((a_{11})_1 - (a_{21})_1)c_1c_2. 
$$

\begin{theorem}
\label{th:2}
The map 
\begin{equation}\label{eq:21}
\rho : G \rightarrow \rm{GL}_2(\Lambda_2), A \mapsto \left(\begin{array}{cc}
1+\beta & \alpha \\
\delta & 1+\gamma \end{array}\right)
\end{equation}
\noindent
ia a homomorphism.
\end{theorem}

\begin{proof}
Let $A'\in G$ and let 
\begin{equation}
\label{eq:22}
R'= R(A')= E + R_2'c_3^2+R_1'c_3+ R_0'+R_{-1}'c_3^{-1},
\end{equation}
\noindent
be decomposition of the form (\ref{eq:16}). Let $\alpha ', \beta ', \gamma ', \delta '$ be the residues of $A'$ and of $R'$ with respect to $c_3$.  

Then
\begin{equation}
\label{eq:23}
\tilde{R} = R(AA') = E + \tilde{R}_2c_3^2+\tilde{R}_1c_3+ \tilde{R}_0+\tilde{R}_{-1}c_3^{-1},
\end{equation}
\noindent
be decomposition of the form (\ref{eq:16}).
Here $\tilde{R}_{-1} = (E+R_0)R_{-1}' + R_{-1}(E+R_0) = (\alpha +\alpha '\beta + \alpha ' +\alpha \gamma ')X.$ Hence the corresponding residue is
\begin{equation}
\label{eq:24}
\tilde{\alpha} = \alpha +\alpha '\beta + \alpha ' +\alpha \gamma '.
\end{equation}
Further, $\tilde{R}_0 = E + R_0 + R_0' + R_0R_0'+ R_1R_{-1}' + R_{-1}R_1'$. Hence 
\begin{equation}
\label{eq:25}
\tilde{R}_0X = (1+ \beta + \beta ' + \beta \beta ' +\alpha \delta ')X. 
\end{equation}
\noindent
Hence 
\begin{equation}
\label{eq:26}
\tilde{\beta} = 1+ \beta + \beta ' + \beta \beta ' +\alpha \delta '
\end{equation}
\noindent
and 
\begin{equation}
\label{eq:27}
\tilde{\gamma} = 1+ \gamma + \gamma ' + \gamma \gamma ' + \delta \alpha  '.
\end{equation}
Then $\tilde{R}_1 = R_1 + R_1' + R_1R_0' + R_0R_1' + R_2R_{-1} + R_{-1}R_2'.$ Hence
 \begin{equation}
\label{eq:28}
\tilde{\delta} = 1+ \delta + \delta ' + \delta \beta ' +\gamma \delta '.
\end{equation}
Consequently,
\begin{equation} 
\label{eq:29}
\left(\begin{array}{cc}
1+\tilde{\beta} & \tilde{\alpha} \\
\tilde{\delta} & 1+\tilde{\gamma} \end{array}\right)\left(\begin{array}{cc}
1+\beta & \alpha \\
\delta & 1+\gamma \end{array}\right) = \left(\begin{array}{cc}
1+\beta ' & \alpha ' \\
\delta '& 1+\gamma '\end{array}\right), 
\end{equation}
equivalent to $\rho (AA') = \rho (A) \rho (A').$
\end{proof}
Now we are to compute  im($\rho$) = $\rho (G).$ Let GL($2, \Lambda_2, \Delta_2$) denote the congruence subgroup of GL($2, \Lambda_2$) with respect to the augmentation ideal $\Delta_2$ of $\Lambda_2$. We denote by GL($2, \Lambda_2, \Delta_2, \Delta_2^2$) 
the subgroup of GL($2, \Lambda_2$) consisting of the matrices corresponding to the following inclusion scheme:
\begin{equation}
\label{eq:30}
 \left(\begin{array}{cc}
1+\Delta_2 & \Lambda_2 \\
\Delta_2^2 & 1+\Delta_2\\
 \end{array}\right).
\end{equation}

\begin{theorem}
\label{th:3}
\rm{Im}($\rho $) = GL($2, \Lambda_2, \Delta_2, \Delta_2^2$). 
\end{theorem}
\begin{proof}
Let 
\begin{equation}
\label{eq:31}
B= \left(\begin{array}{cc}
1+\beta & \alpha \\
\delta & 1+\gamma \end{array}\right)
\end{equation}
\noindent be an 
 invertible matrix corresponding to the inclusion scheme (\ref{eq:30}). Then we have the following  decompositions:
\begin{equation}
\label{eq:32}
\beta = \beta_1c_1 + \beta_2c_2, \gamma = \gamma_1c_1+\gamma_2c_2, \delta = \delta_{11}c_1^2
+\delta_{12}c_1c_2+\delta_{12}'c_1c_2+\delta_{22}c_2^2,
\end{equation}
where $\beta_1, ..., \delta_{22} \in \Lambda_2.$ 

First we define a matrix that stabilizes the column $\bar{c}$ such that  $\rho (C) = B$  
subject to its invertibility. 
\begin{equation}
\label{eq:33}
C= \left(\begin{array}{ccc}
1+\gamma_2c_2+\delta_{12}'c_3 & -\gamma_2c_1+\delta_{22}c_3& -\delta_{12}'c_1-\delta_{22}c_2 \\
-\gamma_1c_2-\delta_{11}c_3 & 1+\gamma_1c_1-\delta_{12}c_3& \delta_{11}c_1+\delta_{12}c_2\\
-\alpha c_2 - \beta_1c_3 & \alpha c_1 - \beta_2c_3 & 1 +\beta_1c_1 +\beta_2c_2\\
 \end{array}\right).
\end{equation}

Obviously, $C\bar{c} = \bar{c}$. By direct computation we obtain that 
\begin{equation}
\label{eq:34}
\mathrm{det}(C) = \mathrm{det}(B) + r, 
\end{equation} 
$$r= \delta_{12}'c_3 + \gamma_1\delta_{12}' c_1c_3 - \delta_{12}c_3 - \delta_{12}'\delta_{12}c_3^2 - \gamma_2\delta_{12}c_2c_3 +
$$
$$
+\beta_2\delta_{12}'c_2c_3 - \beta_1\delta_{12}c_1c_3 - \beta_1\delta_{22}c_2c_3+\gamma_1\delta_{22}c_2c_3 - \gamma_2\delta_{11}c_1c_3 + \delta_{11}\delta_{22}c_3^2 + \beta_2\delta_{11}c_1c_3.
$$
Suppose, that $B\in$ GL($2, \Lambda_1$), then $\beta_2, \gamma_2, \delta_{12}', \delta_{12}, \delta_{22} =0,$ hence $r =0$, and $C\in G.$ Similarly we obtain, that $C \in G$ if $B$ does not depend of $c_1$. 
Since $B = B_1B'$ where $B_1$ does not depend of $c_2$ and $B'$ lies in the congruence subgroup with respect to $c_2$, i.e., 
\begin{equation}
\label{eq:35}
B'\in\left(\begin{array}{cc}
1+\Lambda_2c_2 & \Lambda_2c_2 \\
\Delta_2c_2 & 1+\Lambda_2c_2 \end{array}\right).
\end{equation} 
Both matrices, $B_1$ and $B'$ are invertible, and $B_1\in$ im($G$). There are decompositions (\ref{eq:20}) in which $\beta_1, \gamma_1,  \delta_{11} =0.$ Note that transvection $t=t_{21}((-\delta_{12} - \delta_{12}')c_1c_2$ lies in $B$ and has a preimage in $G.$ Then   
\begin{equation}
\label{eq:36}
B''=B't\in \left(\begin{array}{cc}
1+\Lambda_2c_2 & \Lambda_2c_2 \\
\Lambda_2c_2^2 & 1+ \Lambda_2c_2 \end{array}\right). 
\end{equation}
The elements of $B''$ are decompositions (\ref{eq:20}) such that $\beta_1, \gamma_1, \delta_{12}', \delta_{12}, \delta_{11}=0.$ The corresponding matrix $C''$ defined in the form  (\ref{eq:33}) is invertible because its determinant is equal to det($B''$). Hence $B''\in $ im($\rho$), and $B \in $ im($\rho$). 
\end{proof}
Then $G$ is  an extension of ker($\rho$), that is described by formulas (\ref{eq:20}), by im($\rho$), that is consisting of all invertible matrices corresponding to (\ref{eq:30}). By the way, we note, that ker($\rho$) contains the subgroup $H$ of all  matrices in $G$ of the form 
\begin{equation}
\label{eq:37}
 \left(\begin{array}{ccc}
1+\Lambda_3c_3^2 & \Lambda_3c_3^2& \Lambda_3c_3 \\
\Lambda_3c_3^2 & 1+ \Lambda_3c_3^2& \Lambda_3c_3\\
\Lambda_3c_3^2&\Lambda_3c_3^2&1 + \Lambda_3c_3\\
 \end{array}\right). 
\end{equation}
The quotient im($\rho$)$/H$ is easily understood. 

\section{On the tame stabilizer of a column in GL($3, \Lambda_3$).}

In general, the stabilizer $G$ of $\bar{c_n}$ in GL($n, K$) for any commutative associative ring $K$ with identity 
contains each matrix of the form
\begin{equation}
\label{eq:38}
T_{i, j, k}(a) = E + ac_kE_{ij} - ac_jE_{ik}, \  \rm{for} \   i \not= j, k; j<k; a \in K.
\end{equation}

Also, given the Proposition  \ref{pr:1}, $G$ contains each matrix of the form
\begin{equation}
\label{eq:39}
S_{i,j}(a) = E + ac_ic_jE_{ii} - ac_i^2E_{ij} + ac_j^2E_{ji} - a c_ic_jE_{jj}, \rm{for} \ i < j, a \in K
\end{equation}
\noindent
(see (\ref{eq:9})). 

We denote by $G_t$ ({\it the tame stabilizer}) the subgroup of $G$ generated by  all matrices $T_{i,j,k}(a)$ and $S_{i,j}(a)$ defined by (\ref{eq:38}) and (\ref{eq:39}), respectively. A question arises: Does $G_t$ coincides with $G$? For $n=2$, the answer "Yes"{} is obvious  by Proposition \ref{pr:1}. We'll show below that the answer for  $n=3$ is "No".

Now, let $G\leq$ GL($3, \Lambda_3$) be the stabilizer of $\bar{c}_3$ and let $G_t\leq G$ be the corresponding tame stabilizer. 
As above, $\Lambda_3$ denotes $\Lambda_{30}^{\mathbb{F}}$, 
$\Lambda_{30}^{\mathbb{Z}},$ or $\Lambda_{30}^{\mathbb{Z}}.$ 

We exlude Laurent polynomial rings    
$\Lambda_{33}^{\mathbb{F}}$ over a field. The following results show a connection between $G_t$ and GE($2, \Lambda_2$), that allows to show that $G_t$ is small with respect to $G.$  

\begin{proposition}
\label{pr:2}
 
\rm{im}($G_t$) $\leq $ \rm{GL}($2, \Lambda_2$),
\end{proposition}
\begin{proof}
If the matrix $A=(a_{ij})\in G$ has the form $E + A'$, and all  rows of the matrix $A'$  are zero except for one row, then  $\rho (A)$   lies in the subgroup GE($2, \Lambda_2$). Indeed, formulas (\ref{eq:20}) show that in this case $\alpha = 0$ or $\delta = 0$. Then $\rho (A)$ is a triangular matrix. But every triangular matrix lies obviously  in GE($2, \Lambda_2$). This proves the statement for any matrix $A = T_{i,j,k}(a).$

By formulas (\ref{eq:20}) for any matrix $S_{i,j}(a)$, one has  $\alpha = 0,$ 
and we conclude as above. 
\end{proof}
\begin{theorem}
\label{th:4}
Let $G$ be the stabilizer of the column $\bar{c}_3$ in 
GL($3, \Lambda_{33}^{\mathbb {Z}}$).  Then for every finite subset $L \subseteq G$
\begin{equation}
\label{eq:40}
gp(L, G_t) \not= G.
\end{equation}
\end{theorem}

\begin{proof}
By Bachmuth and Mochizuki result \cite{BM}, if $n\geq 2$ then GL($2, \Lambda_n^{\mathbb{Z}})$ can not be generated by any finite subset together with the subgroup GE($2, \Lambda_n^{\mathbb{Z}}.$ 

Hence, 
\begin{equation}
\label{eq:41}
gp(\rm{GE}(2, \Lambda_{22}^{\mathbb {Z}}), \rho (L)) \not= \rm{GL}(2,\Lambda_{22}^{\mathbb {Z}}). 
\end{equation}
Then there is a matrix $A$ that belongs to the difference between the two sides of (\ref{eq:41}). We'll show that there is a similar matrix with elements corresponding to the scheme (\ref{eq:30}). To prove this assertion, we define the image  $E+A_0$ of $A$ that lies in GL($2, \mathbb{Z}$)
under specialization  homomorphism
GL($2,\Lambda_{22}^{\mathbb{Z}}$)$\rightarrow$ GL($2, \mathbb{Z}$) defined by the map 
$c_i\mapsto 1, i = 1, 2$.

In  other words, $A_0$ is  the $0$ part of $A$ under the decomposition form (\ref{eq:16}). Then $E+A_0 \in$ GE($2, \mathbb{Z}).$ We multiply $A$ by $(E+A_0)^{-1}$ and get new matrix $\tilde{A}$ with the same property. Suppose, that  its (21) component $\tilde{a}_{21}= q_1c_1+q_2c_2 + q_3,$ where $q_1, q_2\in \mathbb{Z}$, 
$q_3\in \Delta_2^{\mathbb{Z}}$ does not lie in $\Delta_{22}^{\mathbb{Z}}.$  This means that $q_1\not= 0$ or $q_2\not=0.$ Then we multiply $\tilde{A}$ by $t_{21}(-q_1c_1- q_2c_2)$ and obtain matrix $\bar{A}$ that lies in the difference the two sides   of (\ref{eq:41}) and corresponds to the scheme (\ref{eq:30}). Thus $\bar{A}\in $ im($\rho$) but has no preimages in gp($\rm{GE}(2, \Lambda_{22}^{\mathbb {Z}}), \rho (L)$). 
\end{proof}

\section*{Conclusion}

The main results of this paper were obtained for matrix groups over polynomial rings under fairly rigorous assumptions regarding a stabilized vector. The similar results can be obtained for other rings. For example, Theorem 1 can be proved for any polynomial ring over a commutative associative ring with identity $K$ over one variable $x$ for the corresponding stabilized vector $c = \left(\begin{array}{c}k_1\\ k_2\\ x\\ \end{array}\right)$, where $k_1, k_2$ are arbitrary elements of $K.$ The main advantage of the proposed method is the fact that  we move by the homomorphism $\rho$ from matrices over the module $\Lambda_3 + c_3^{-1}\Lambda_3$ to matrices over $\Lambda_2$. This process can be considered as  an elimination of the residue $c_3^{-1}.$ Such moving in a number of cases allows the using of the corresponding induction. This approach also demonstrates the parallelism of theories of groups of automorphisms of groups and matrix groups that exists for a number of well-known groups. This allows us to use the results on matrix groups to describe automorphism groups.

\bigskip
{ROMAN'KOV V. A.,} {Chief Researcher,  Sobolev Institute of Mathematics SB RAS, Novosibirsk, Russia}

{romankov48@mail.ru}

\end{document}